\documentclass{svproc}
\usepackage{url}

\usepackage{graphicx}
\usepackage{amsmath}
\usepackage{amssymb}
\usepackage{placeins}
\usepackage{tikz}
\usepackage[hidelinks]{hyperref}
\usepackage[all]{hypcap}

\newcommand{\orcidicon}[1]{%
  \href{https://orcid.org/#1}{%
    \raisebox{-0.2ex}{\includegraphics[height=1.6ex]{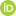}}%
  }%
}
\begin{document}
\mainmatter              
\title{Detecting Zariski Pairs by Algorithms and Computational Classification
in Conic--Line Arrangements}
\titlerunning{Detecting Zariski Pairs in Conic--Line Arrangements}  
%
\author{%
Meirav Amram\inst{1}\,\orcidicon{0000-0003-4912-4672}
\and
Gal Goren\inst{2}\,\orcidicon{0009-0007-8033-3663}%
}
\authorrunning{Amram et al.} 
%
\tocauthor{Meirav Amram, Gal Goren}
\institute{Department of Mathematics, Shamoon College of Engineering, Ashdod, Israel\\
\and
Faculty of Mathematics, Technion - Israel Institute of Technology, Haifa, Israel
\email{meiravt@sce.ac.il, lag@campus.technion.ac.il (corresponding author)}
}

\maketitle              

\begin{abstract}
We present an approach to detecting Zariski pairs in conic–line arrangements. Our method introduces a combinatorial condition that reformulates the tubular-neighborhood homeomorphism criterion arising in the definition of Zariski pairs. This allows for a classification of arrangements into combinatorial equivalence classes, which we generate systematically via an inductive algorithm. For each class, potential Zariski pairs are examined using structural lemmas, projective equivalence, and fundamental group computations obtained through the Zariski--van Kampen Theorem.
\end{abstract}
%
\section{Introduction}
The study of Zariski pairs originated with Zariski \cite{zariski29} and has led many researchers to investigate this phenomenon.

Before we present the objectives of our study and the results, we provide the definition of the Zariski pair as presented in \cite{Bartolo_tokunaga_survey}.
\begin{definition}\label{def:zariski_pairs}
A pair of reduced plane curves $ \mathcal{B}_1,\mathcal{B}_2 \subseteq \mathbb{CP}^2 $ is called a \emph{Zariski pair} if:
\begin{enumerate}
	\item There exist tubular neighborhoods $ T(\mathcal{B}_1) $ and  $ T(\mathcal{B}_2) $ and a homeomorphism
	$ h: T(\mathcal{B}_1)\to T(\mathcal{B}_2) $ with $ h(\mathcal{B}_1)=\mathcal{B}_2 $.
	
	\item There is no homeomorphism $ f:\mathbb{CP}^2\to \mathbb{CP}^2 $ with $ f(\mathcal{B}_1)=\mathcal{B}_2 $.
\end{enumerate}
\end{definition}

We denote the degree of a reduced plane curve as the sum of the degrees of its irreducible components.
In this paper, when we say that a pair of reduced plane curves $(\mathcal{B}_1,\mathcal{B}_2)$ is homeomorphic, we mean there is a homeomorphism $f : \mathbb{CP}^2 \rightarrow \mathbb{CP}^2$ such that $f(\mathcal{B}_1) = \mathcal{B}_2$.
The main objective of this paper is to present an approach to detecting Zariski pairs within conic--line arrangements.
We approach this problem by first introducing a combinatorial condition, which serves as a combinatorial reformulation of the tubular-neighborhood homeomorphism condition appearing in the definition of Zariski pairs. This enables us to divide all possible conic–line arrangements into equivalence classes of combinatorial type. Using an inductive program, we systematically generate all such classes for the configurations considered here. For each resulting combinatorial type, we then investigate the existence of Zariski pairs by applying a combination of structural lemmas, projective equivalence arguments, and fundamental group computations. By the Zariski--van Kampen Theorem \cite{vanKampen33} we can derive presentations of those groups by means of generators and relations.
This technique is very strong in the sense that many of the known examples of Zariski pairs have distinct fundamental groups of their complements.

Before continuing, we provide an overview of existing results and current studies in this domain.
Zariski presented the first Zariski pair as two sextics with 6 cusps each, where the cusps lie on a conic in the first curve, but not in the second curve \cite{zariski29}.
 Since then, finding Zariski pairs has been developed, especially by Artal-Bartolo, Cogolludo-Agustin, and Tokunaga in \cite{Bartolo_tokunaga_survey}, and by Artal-Bartolo in \cite{Bartolo94}. Artal-Bartolo pointed out that two curves with isomorphic combinatorics and different topologies are a Zariski pair.
 Shimada extended Artal-Bartolo's method of Kummer coverings from \cite{Bartolo94}, to construct two infinite series of Zariski pairs of increasing degrees in \cite{shimada}.

In \cite{DEG_isotopy}, Degtyarev proves that the rigid isotopy class of curves of degree at most 5
is determined by the combinatorial data. In particular, his work implies that, up to
degree 5, there are no Zariski pairs. He also classifies in \cite{DEG} the isotopy classes of the
curves of degree 6 with simple singularities.

In \cite[Appendix A]{shustin_appendix}, Shustin used patchworking construction to construct plane curves of degree $ \nu (\nu-1) $ (for $ 3\le \nu \le 10 $), which cannot appear as branch loci of degree $ \nu $ surfaces, yielding examples of Zariski pairs of degrees $ d\in \{ 6, 12, 20, 30, 42, 56, 72, 90 \} $. 
Bannai and Tokunaga \cite{bannai_tokunaga2019zariski}, Artal-Bartolo, Bannai, Shirane, and Tokunaga \cite{bartolo_tokunaga2020torsion,bartolo_tokunaga2020torsion2}, and Takahashi and Tokunaga \cite{takahashi_tokunaga2020explicit} used the arithmetic of elliptic curves to construct examples of Zariski pairs and Zariski tuples. 

Oka directly compared the fundamental groups $ \pi_1(\mathbb{CP}^2 - \mathcal{B}_1, *)$ and $ \pi_1(\mathbb{CP}^2 - \mathcal{B}_2, *)$ of the complements of $\mathcal{B}_1$ and $\mathcal{B}_2$ in \cite{Oka1999FlexCA,Oka2002}, which is also Zariski's original point of view on the problem. 

Artal-Bartolo and Dimca \cite{Dimca} studied the fundamental groups of plane curves, describing topological properties of curves with an abelian fundamental group. Artal-Bartolo, Carmona, and Cogolludo-Agustin \cite{Artal2} showed that the braid monodromy of an affine plane curve determines the topology of a related projective plane curve, and that the Zariski–-van Kampen Theorem can be used to compute its fundamental group. Artal-Bartolo, Carmona, Cogolludo-Agustin, Luengo, and Melle \cite{Artal3} further reviewed the Zariski-–van Kampen method as a key tool for obtaining presentations of fundamental groups of curve complements in $\mathbb{P}^2$. Namba and Tsuchihashi, in \cite{Namba2004OnTF},  used the fundamental group of the complement to construct an example of a Zariski pair of a degree 8 conic--line arrangement; Tokunaga used the theory of dihedral covers to construct a  few such examples of degree 7 in \cite{tokunaga2014}. 

\section{Conventions, examples and motivation}

In this work, we focus on algebraic curves that are conic--line arrangements. This framework naturally generalizes the study of line arrangements and allows existing techniques to be extended to a broader class of configurations.

\subsection{Conventions}
\begin{definition}
\leavevmode
\begin{enumerate}
    \item An $(n,m)$--\emph{arrangement} is a reduced conic--line arrangement consisting of $n$ lines and $m$ conics in $\mathbb{CP}^2$.
    \item A \emph{constraint} is a point of intersection of at least three components (lines or conics); such points play a central role in our combinatorial classification.
    \item A Zariski pair $(\mathcal{B}_1,\mathcal{B}_2)$ is \emph{minimal} if no proper sub-arrangements $\mathcal{B}_1',\mathcal{B}_2'$ form a Zariski pair.
\end{enumerate}
\end{definition}
\begin{example}
  Following the conventions above, we present an example of a $(4,1)$--arrangement. See Fig.~\ref{41example}.
  \end{example}
    \begin{figure}[ht]

\begin{center}
\label{41example}
\includegraphics[width=\linewidth]{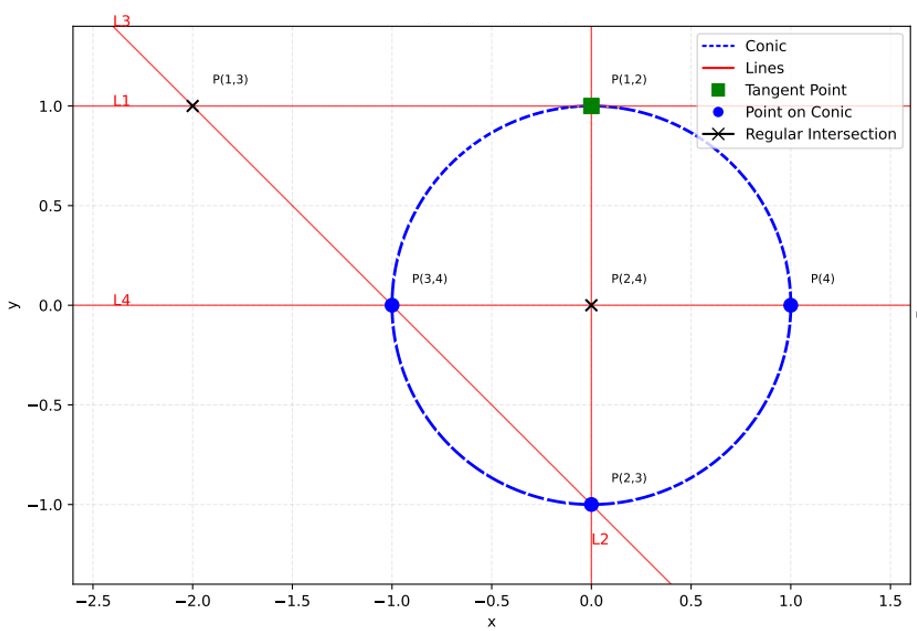}
\end{center}
\caption{Example of a $(4,1)$--arrangement}
\end{figure}
\subsection{Known examples}

Most known Zariski pairs consist of irreducible curves. But we can try to find Zariski pairs among curves at the other extreme - in the unions of curves of low degree. 

A meaningful study has been done on conic--line arrangements, as we show in the following examples.

\begin{example}
    Tokunaga's arrangements are a very good example for Zariski pairs. Tokunaga proved in \cite{tokunaga2014} that the two pairs of conic--line arrangements in Figs.~\ref{fig_tokunaga1} and~\ref{fig_tokunaga2} are Zariski pairs. Then, in \cite{arxiv_version}, this result was proven via fundamental groups as well.
    
    The two upper conic--line arrangements are $(3,2)$--arrangements. One arrangement can be obtained from the other by continuously rotating one of the lines. The types of singularities are invariant under this process. 

    The two bottom $(1,3)$-arrangements also form a Zariski pair. Each curve consists of three conics and one line, and again one can be obtained from the other by a continuous rotation of the line.


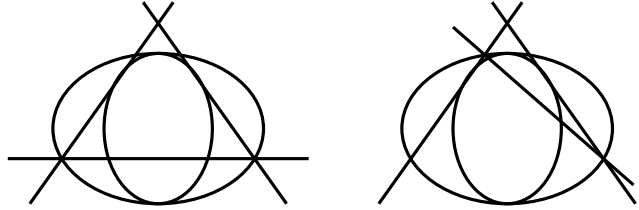
\begin{figure}[ht]
\centering
\begin{tikzpicture} [scale=0.2, style = very thick]
	\begin{scope}[xshift=-330]
		\draw (0,0) ellipse (7 and 5);
		\draw (8,0) node {\footnotesize };
		\draw (0,0) ellipse (3.6 and 5);
		\draw (3,0) node {\footnotesize };
		\draw (-10, -2) -- (10, -2);
		\draw (11, -2) node {\footnotesize };
		\draw (-1,8.406) -- (8.533, -5);
		\draw (8.6, -6) node {\footnotesize };
		\draw (1,8.406) -- (-8.533, -5);
		\draw (-8.6, -6) node {\footnotesize };
		
		\draw (0, -4) node {\footnotesize };
		\draw (0, 4) node {\footnotesize };
		
		\draw (7.3, -2.5) node {\footnotesize };
		\draw (-7.4, -2.5) node {\footnotesize };
		\draw (2.2, 5.5) node {\footnotesize };
		\draw (-2.3, 5.5) node {\footnotesize };
		
		\draw (0, -7) node {};
	\end{scope}
	
	\begin{scope}[xshift=330]
		\draw (0,0) ellipse (7 and 5);
		\draw (8,0) node {\footnotesize };
		\draw (0,0) ellipse (3.6 and 5);
		\draw (3,-0.5) node {\footnotesize };
		\draw (-1,8.406) -- (8.533, -5);
		\draw (8.6, -6) node {\footnotesize };
		\draw (1,8.406) -- (-8.533, -5);
		\draw (-8.6, -6) node {\footnotesize };
		\draw (8.4, -3.75) -- (-3.6, 6.75);
		\draw (9, -4) node {\footnotesize };
		
		\draw (0, -4) node {\footnotesize };
		\draw (0, 5.7) node {\footnotesize };
		
		\draw (7.4, -2.2) node {\footnotesize };
		\draw (-7.4, -2.5) node {\footnotesize };
		\draw (2.2, 5.5) node {\footnotesize };
		\draw (-1.9, 5.7) node {\footnotesize };
		
		\draw (0, -7) node {};
	\end{scope}
	
\end{tikzpicture}
\caption{A Zariski pair with two conics and 3 lines}\label{fig_tokunaga1}
\end{figure}

\begin{figure}[ht]
\centering
\begin{tikzpicture} [scale=0.2, style = very thick]
	\begin{scope}[xshift=-330]
		\draw (0,0) ellipse (7 and 4);
		\draw (-7, -3) node {\footnotesize };
		
		\draw (0,0) ellipse (4 and 7);
		\draw (-3, 7) node {\footnotesize };
		
		\draw (0,0) ellipse (4 and 4);
		\draw (-2, 2) node {\footnotesize };
		
		\draw (3.47, -7) -- (3.47, 7); 
		\draw (4, -6) node {\footnotesize };
		
		\draw (-4.7, 0) node {\footnotesize };
		\draw (4.7, 0) node {\footnotesize };
		\draw (0, 4.7) node {\footnotesize };
		\draw (0, -4.7) node {\footnotesize };
		
		\draw (4.4, 3.7) node {\footnotesize };
		\draw (-4.4, 3.7) node {\footnotesize };
		\draw (4.4, -3.7) node {\footnotesize };
		\draw (-4.4, -3.7) node {\footnotesize };
		\draw (0, -8) node {};
	\end{scope}
	
	\begin{scope}[xshift=330]
		\draw (0,0) ellipse (7 and 4);
		\draw (-7, -3) node {\footnotesize };
		
		\draw (0,0) ellipse (4 and 7);
		\draw (-3, 7) node {\footnotesize };
		
		\draw (0,0) ellipse (4 and 4);
		\draw (-2, 2) node {\footnotesize };
		
		\draw (-7, -7) -- (7, 7); 
		\draw (-4, -6) node {\footnotesize };
		
		\draw (-4.7, 0) node {\footnotesize };
		\draw (4.7, 0) node {\footnotesize };
		\draw (0, 4.7) node {\footnotesize };
		\draw (0, -4.7) node {\footnotesize };
		
		\draw (4.4, 3.7) node {\footnotesize };
		\draw (-4.4, 3.7) node {\footnotesize };
		\draw (4.4, -3.7) node {\footnotesize };
		\draw (-4.4, -3.7) node {\footnotesize };
		
		\draw (0, -8) node {};
	\end{scope}
\end{tikzpicture}

\caption{A Zariski pair with three conics and one line}
\label{fig_tokunaga2}
\end{figure}
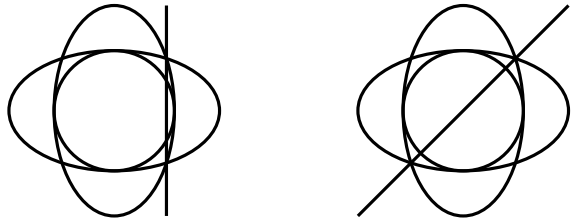
\end{example}

There is also a study on conic arrangements that we think is worthwhile to give here as a motivation to study conic--line arrangements.

\begin{example}\label{namba}
Namba and Tsuchihashi used in \cite{Namba2004OnTF} the fundamental group of the complement to construct an example of a Zariski pair of degree 8 conic arrangements, see Fig.~\ref{fig:Namba}.

 \begin{figure}[ht]
     \centering
    \includegraphics[width=0.7\linewidth]{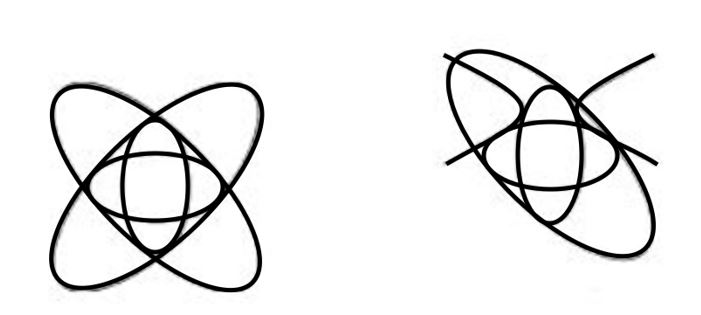}
     \caption{A Zariski pair of conic arrangements of degree 8}
       \label{fig:Namba}
 \end{figure}
 
\end{example}
Results like these, as we see in Example \ref{namba}, can be very helpful when searching for $(n,m)$--arrangements with $m \geq$ 4.

A joint work of the first author with Shwartz, Sinichkin, Tan, and Tokunaga gives a new Zariski pair of conic--line arrangements, as explained in the following example.

\begin{example}
A new Zariski pair of conic--line arrangements with three conics and two lines was found in \cite{arxiv_version}. 
Each curve is a $(2,3)$--arrangement. They both consist of Tokunaga's $(1,3)$--arrangements, see Fig.~\ref{degree8pair}.
\begin{figure}[ht]
\centering
\includegraphics[width=0.8\linewidth]{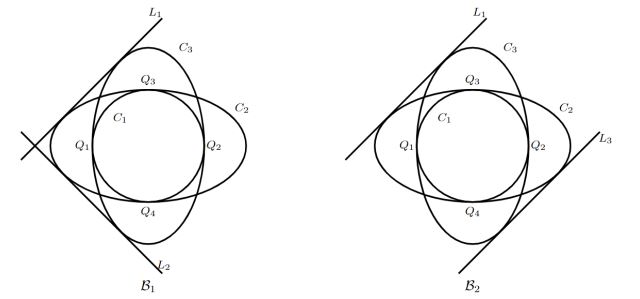}
\caption{A new Zariski pair of $(2,3)$--arrangements}\label{degree8pair}
\end{figure}
\end{example}

\FloatBarrier
\subsection{Motivation}

Our aim is to discover a significant number of new examples, or at least to lay the foundation for a method that can ultimately lead to a complete classification of $(n,m)$--arrangements. 

In the realm of $(n,1)$--arrangements, relatively little is known. Only recently has the first example of a Zariski pair of this type been discovered. Motivated by the broader goal of achieving a complete classification of Zariski pairs in conic–line arrangements, we begin by examining the cases of $(n,1)$--arrangements. The first such example was identified by Bannai, Guerville-Ballé, and Shirane in \cite{ben2}, where they construct Zariski pairs arising from  $(7,1)$--arrangements, see Fig.~\ref{71example}.

\begin{figure}[ht] 
\centering
\includegraphics[width=\linewidth]{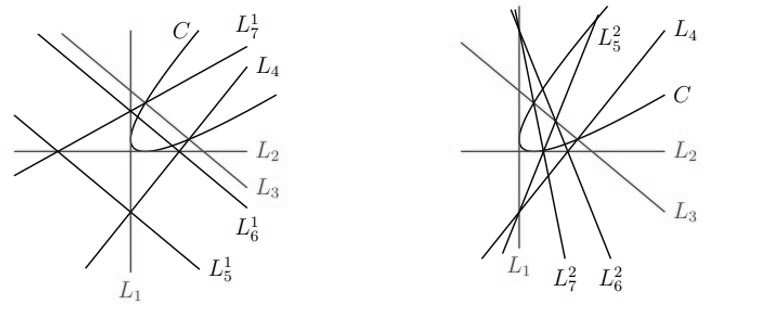}
\caption{$(7,1)$--arrangement discovered in \cite{ben2}.}
\label{71example}
\end{figure}

\FloatBarrier

Thus far, only isolated examples of conic--line Zariski pairs have been found, and no systematic approach exists for generating or classifying them. In the next section, we present our computational framework, which inductively generates $(n,1)$--arrangements and tests them for equivalence.

\section{Computations and results}
We introduce a notion of combinatorial equivalence that reflects the first condition in the definition of a Zariski pair. This combinatorial framework serves as the foundation for an inductive program that generates all $(n,1)$–arrangements up to equivalence, forming the pool from which potential Zariski pairs may later be identified.

\subsection{Combinatorial equivalence}

Let $C$ be an $(n,1)$--arrangement and let $I(C)$ be the set of all intersection points in $C$.
\begin{definition}
    For each $p \in I(C)$ we associate a \emph{characteristic triple}
\[
    \operatorname{char}(p) = (\chi_1,\chi_2,\chi_3),
\]
defined as follows:
\begin{enumerate}
    \item $\chi_1 = 1$ if $p$ is a tangency point, and $\chi_1 = 0$ otherwise.
    \item $\chi_2 = 1$ if $p$ lies on the conic, and $\chi_2 = 0$ otherwise.
    \item $\chi_3 =$ the number of distinct lines through $p$ which are not tangent at $p$.
\end{enumerate}

\end{definition}
\begin{theorem} \label{thrm32}
    For a pair $(C_1,C_2)$ of $(n,1)$--arrangements, the following are equivalent:
    \begin{enumerate}
    \item There exist tubular neighborhoods $T(C_1),T(C_2)$ and a homeomorphism 
    $h : T(C_1) \to T(C_2)$ with $h(C_1)=C_2$.
    \item There exist indexings of the lines of $C_1$ and $C_2$, respectively, and a bijection
    $b : I(C_1) \to I(C_2)$ such that for every $p \in I(C_1)$:
   \begin{itemize}
        \item $\operatorname{char}(b(p)) = \operatorname{char}(p)$
        \item the set of line indices through $p$ is the same as the set of line indices through $b(p)$,
              and if a line is tangent at $p$ then the line with the same index is tangent at $b(p)$.
    \end{itemize}
\end{enumerate}
\end{theorem}

The multiset of all characteristic triples of intersection points in $C$ is denoted as the weak numerical type. 
\begin{example}
Fig.~\ref{fig7}. shows all $(3,1)$–arrangements up to combinatorial
equivalence. For $(n,1)$–arrangements with $n \le 3$, the weak numerical type
already determines the combinatorics.

Consider specifically the upper-left example in Fig.~\ref{fig7}. This
$(3,1)$–arrangement consists of a conic and three tangent lines. There are two
kinds of intersection points: tangency points (denoted $P_1$) and simple
intersection points between lines (denoted $P_2$). For any $p_1 \in P_1$ we have
$\operatorname{char}(p_1) = (1,1,0)$, and for any $p_2 \in P_2$ we have
$\operatorname{char}(p_2) = (0,0,2)$.
\end{example}

\begin{figure}[ht]
\centering
\scalebox{0.67}{\includegraphics{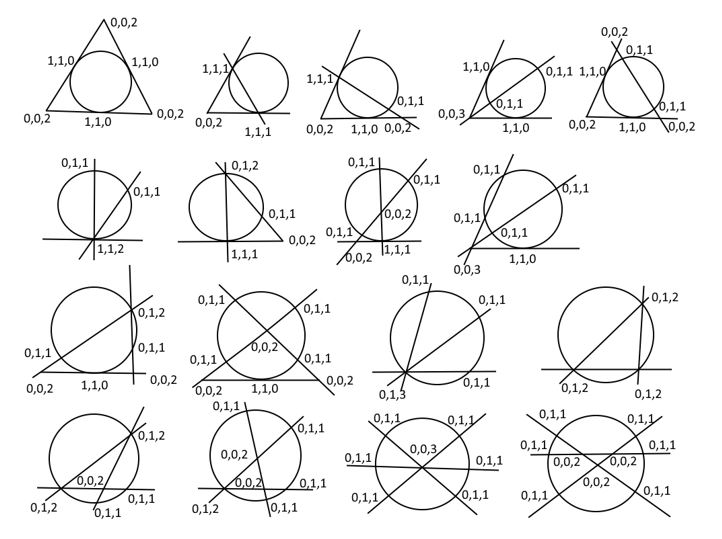}}
\caption{All $(3,1)-$arrangements up to this combinatorial equivalency}\label{fig7}
\end{figure}


\subsection{Algorithmic approach}

We develop an \emph{inductive program} to generate and test all possible $(n,1)$--arrangements:

\begin{enumerate}
    \item Start with the unique $(0,1)$--arrangement, consisting only of a conic.
    \item At each step $j$, generate all $(j,1)$--arrangements by adding a new line in every possible way to each $(j-1,1)$--arrangement.
    \item After each step, remove duplicates up to combinatorial equivalence.
\end{enumerate}

Repeating this procedure for $j \leq n$ yields the complete collection of $(n,1)$--arrangements up to combinatorial equivalence.

\subsection{Methods for identifying Zariski pairs}
To determine whether a pair of arrangements truly forms a Zariski pair, we apply the following methods:
\begin{itemize}
    \item Applying known and new lemmas about minimal Zariski pairs.
    \item Testing for projective equivalence via the action of $PGL(3,\mathbb{C})$.
    \item Computing fundamental groups of the complements. 
\end{itemize}
\noindent\textbf{Key lemmas.}\label{sec:keylemmas}
Let $(\mathcal{B}_1,\mathcal{B}_2)$ be a minimal Zariski pair of $(n,1)$--arrangements and let $L\in B_1$ be a line. We now state key minimality criteria that allow us to eliminate many combinatorial candidates without performing group computations.

\begin{lemma}\label{lemma_1}
Line $L$ must pass through at least one constraint.
\end{lemma}

\begin{proof}
Assume $L$ meets no constraint. Let $L'$ be the line in $B_2$ corresponding to $L$ under the fixed combinatorial identification between $B_1$ and $B_2$. 
Set $C_1=\overline{B_1\setminus L}$ and $C_2=\overline{B_2\setminus L'}$. 
By construction, $C_1$ and $C_2$ have the same combinatorics. Since $(B_1,B_2)$ is \emph{minimal}, $C_1$ and $C_2$ are homeomorphic. So we can assume without the loss of generality that $C_1 = C_2$. Now, we can easily find a continuous deformation that takes $L$ to $L'$, leaving the combinatorics constant throughout the deformation. Hence $B_1 \simeq B_2$, a contradiction.
\end{proof}

\begin{lemma}\label{lemma2}
If a non-tangent line $L$ has exactly one constraint $p$, then the characteristic of $p$ in 
$\overline{B_1 \setminus L}$ occurs at least at one other point in $\overline{B_1 \setminus L}$.
\end{lemma}

\begin{proof}
Assume for contradiction that the characteristic of $p$ is unique in 
$\overline{B_1 \setminus L}$.  
Let $L'$ be the line in $B_2$ corresponding to $L$, so that 
$\overline{B_1 \setminus L} \simeq \overline{B_2 \setminus L'}$.  
Since the characteristic of $p$ in $\overline{B_1 \setminus L}$ is unique, and $L$ is not tangent to the conic, there is a unique way, up to a continuous deformation that preserves the combinatorics, to add a line to $\overline{B_1 \setminus L}$ to recover 
the combinatorics of $B_1$.  
Thus
\[
B_1 = \overline{B_1 \setminus L} \cup L \;\;\simeq\;\; 
\overline{B_2 \setminus L'} \cup L' = B_2,
\]
which is a contradiction.
\end{proof}
\begin{remark}
  If $L$ is tangent to the conic, then the characteristic of its unique constraint point does not determine the direction of tangency. Consequently, the process of adding back the removed line is not unique (even up to isotopy), so the above argument does not apply.  
\end{remark}

\begin{lemma}
If line $L$ has exactly two constraints $p, q$, then the characteristics of $p$ and $q$ in 
$\overline{B_1 \setminus L}$ occur at least at one other pair of points in $\overline{B_1 \setminus L}$.
\end{lemma}

\begin{proof}
    The argument is the same as in Lemma \ref{lemma2}.
\end{proof}


\medskip
\noindent\textbf{Projective equivalence.}\label{sec:pgl}
We use the $3$--transitivity of $PGL(3,\mathbb{C})$ to test whether two configurations are projectively equivalent.  

If such a transformation exists, the arrangements are projectively (hence topologically) equivalent.

\medskip

\noindent\textbf{Fundamental groups.}\label{sec:groups}
We compute fundamental groups of complements to further distinguish arrangements.  
If two arrangements share the same combinatorics but have non-isomorphic fundamental groups, they form a Zariski pair.  
Even when the groups are isomorphic, we may still have a Zariski pair.

Isomorphism tests rely on algebraic algorithms, developed by Gebhardt, which allow us to compare group presentations rigorously.

\subsection{Achievements to date}

\begin{itemize}
    \item We proved that no $(n,1)$--arrangement Zariski pairs exist for $n \leq 4$.
    \item Our program classifies all $(n,1)$--arrangements for $n \leq 5$.
    \item New lemmas were discovered to eliminate potential pairs.
    \item Algorithms were developed to analyze fundamental groups and detect Zariski pairs.
\end{itemize}

\subsection{Revisiting conic-line arrangement example}

Let us revisit the example of a $(4,1)$--arrangement, see  \hyperlink{41example}{Fig.~1}. Denote this arrangement as $C$. Its combinatorial type is given in the following table.
\FloatBarrier
\begin{table}[h!]
\centering
\begin{tabular}{|c|l|}
\hline
\textbf{Point} & \textbf{Characteristic} \\ \hline
$P(1,2)$ & characteristic is (1, 1, 1), lines are 1, 2, tangency is 1 \\ \hline
$P(4)$ & characteristic is (0, 1, 1), lines are 4\\ \hline
$P(2,4)$ & characteristic is (0, 0, 2), lines are 2,4 \\ \hline
$P(2,3)$ &characteristic is (0, 1, 2), lines are 2, 3 \\ \hline
$P(3,4)$ & characteristic is (0, 1, 2), lines are 3, 4\\ \hline
$P(1,3)$ &characteristic is (0, 0, 2), lines are 1, 3 \\ \hline
$P(1,4)$ & characteristic is (0, 0, 2), lines are 1, 4\\ \hline
\end{tabular}
\caption{Points in Fig.~\ref{41example}. and their characteristics}
\end{table}
\FloatBarrier
Observe that $L_4$ passes through exactly one constraint, $P(3,4)$. If we remove $L_4$, then the characteristic of $P(3,4)$ in $\overline{C \setminus L_4} $ is unique. By Lemma~\ref{lemma2}, there are no Zariski pairs of this combinatorial type.

\subsection{Summary}

We have:
\begin{itemize}
    \item Defined Zariski pairs and their role in the study of algebraic curves.
    \item Transformed the problem of finding all Zariski pair candidates from a topological to a combinatorial one.
    \item Presented the inductive algorithm used by our program to generate all possible $(n,1)$--arrangements for certain natural $n$.
    \item Discussed the methods used to determine whether a candidate is truly a Zariski pair, including projective equivalence and fundamental groups.
    \item Reported current results, notably the claim that there are no Zariski pairs among $(4,1)$--arrangements.
    \item Revisited a specific $(4,1)$--arrangement and demonstrated how our combinatorial criterion and lemmas can be applied to rule out the existence of Zariski pairs in this case.

\end{itemize}

\subsection{Next steps}

We outline two main directions for continuing this research.
\paragraph{More lines}
As the number of lines increases, the complexity of the arrangement space grows rapidly. New intersection patterns arise that require additional treatment beyond the techniques used in low-degree cases. A key challenge is the combinatorial explosion in the number of candidate configurations. We therefore plan to optimize our algorithm and refine the process of adding lines in order to efficiently handle larger values of~$n$.

\paragraph{Multiple conics}
The presence of more than one conic significantly increases the number of geometric relationships and intersection types that must be considered. The current definition of combinatorial equivalence must be extended to incorporate interactions between several conics and between conics and lines. We plan to improve the algorithmic efficiency of our methods to manage the increased number of cases and comparisons required in this broader setting.

\section*{Acknowledgement}
This work was carried out in collaboration with Uriel Sinichkin (Tel Aviv University) and Volker Gebhardt (Western Sydney University). 
We thank them for their guidance, mathematical insights, and many helpful discussions throughout the development of the project.

%

%
%
\bibliographystyle{spmpsci}
\bibliography{references}

\end{document}